\documentclass[11pt,a4paper]{article}
\usepackage[T1]{fontenc}
\usepackage[utf8]{inputenc}
\usepackage{lmodern,microtype,amsmath,amssymb,amsthm}
\usepackage[margin=27mm]{geometry}
\usepackage{enumitem,needspace}
\usepackage[hidelinks]{hyperref}
\hypersetup{pdftitle={Conway's thrackle conjecture: A new upper bound of $4n/3$.},pdfauthor={Dániel Gábor Simon}}
\setlist[enumerate]{itemsep=2pt,topsep=3pt}
\newtheorem{theorem}{Theorem}
\newtheorem{lemma}[theorem]{Lemma}
\newtheorem{proposition}[theorem]{Proposition}
\newtheorem{pammertheorem}{Theorem}

\title{Conway's thrackle conjecture: a new upper bound of $4n/3$.}
\author{D\'aniel G\'abor Simon\thanks{ORCID: \url{https://orcid.org/0000-0002-3750-9666}}\\[4pt]
\small HUN-REN Alfr\'ed R\'enyi Institute of Mathematics,\\
\small ELTE E\"otv\"os Lor\'and University,\\[2pt]
\small Corresponding author: \href{mailto:dgsimon@renyi.hu}{\texttt{dgsimon@renyi.hu}}}
\date{}

\begin{document}
\maketitle

\begin{abstract}
A thrackle is a drawing of a graph in the plane in which every pair of
edges meets exactly once, either at a common endpoint or at a proper
crossing. Conway conjectured that every thrackle has at most as many
edges as vertices. We give a short proof that every thrackle on
$n\ge4$ vertices has at most $\lfloor4(n-1)/3\rfloor$ edges.
The proof combines Pammer's computer-assisted verification of the conjecture for graphs on at most twelve vertices with a discharging argument in suitable plane and projective-plane embeddings. We also give bipartite plane graphs showing
that the coefficient $4/3$ is best possible under the local hypotheses
used in the proof.
\end{abstract}


\section{Introduction}

Throughout the paper, we consider finite simple graphs $G$. In a drawing
of $G$ in the plane, vertices are distinct points, and each edge is a
simple arc joining its endpoints and containing no other vertex. A drawing
is called simple if any two edges meet at most once and every intersection
of edge interiors is a proper crossing. Drawings are in general position,
so no three edge interiors meet at a point.
A simple drawing is called a \emph{thrackle} if any two distinct edges
share exactly one point: either their common endpoint or a proper
crossing. A graph is \emph{thrackleable} if it admits such a drawing.
We write $n=|V(G)|$ and $m=|E(G)|$. Conway's thrackle conjecture,
recorded by Woodall~\cite{Woodall1971}, asserts that $m\le n$ for
every thrackleable graph.

The first linear upper bound
was proved by Lov\'asz, Pach, and Szegedy~\cite{LPS1997}, who showed that
$m\le2n-3$ for $n\ge3$. Cairns and Nikolayevsky~\cite{CN2000} improved
this to $m\le\frac32(n-1)$ using properties of generalized thrackles.
Later, Fulek and Pach~\cite{FP2011} introduced a computational method
for improving the upper bound and proved $m\le\frac{167}{117}n$.
The following bounds are stated for $n\ge4$, to exclude the exceptional
triangle. Xu~\cite{Xu2014} obtained $m\le1.4n$ by a short discharging
argument; the published result of Goddyn and Xu~\cite{GX2017} gives
$m\le1.4(n-1)$. Refining the discharging method, Fulek and
Pach~\cite{FP2019} proved $m\le1.3984n$, and Xu~\cite{Xu2021}
subsequently proved $m\le1.393(n-1)$.
T\'oth~\cite{Toth2026} and, independently, Scoppetto~\cite{Scoppetto2026}
obtained the bound $m\le\frac{18}{13}(n-1)$.

\Needspace{7\baselineskip}
We combine Pammer's computer-assisted verification for graphs on at most
twelve vertices with a short discharging argument to prove the following
bound.

\begin{theorem}\label{thm:1}
Let $G$ be a thrackleable graph on $n\ge4$ vertices and $m$ edges. Then
\[
 m\le\left\lfloor\frac{4(n-1)}3\right\rfloor .
\]
\end{theorem}

The same upper bound is also stated in a contemporaneous
preprint of Tomar~\cite{Tomar}, using separate computer searches
and a discharging argument. The current manuscript is independent work, and its first version was published on ArXiv the same day as Tomar's manuscript on Zenodo. 
Pammer verified Conway's conjecture for graphs
on at most twelve vertices in his 2014 master's thesis:

\begin{pammertheorem}[Pammer~{\cite[Observation~5.16, p.~54]{Pammer2014}}]
\label{thm:pammer}
Any thrackle on at most $12$ vertices contains at most as many edges as
vertices.
\end{pammertheorem}

We use Theorem~\ref{thm:pammer} first to restrict intersections of small
cycles (Section~\ref{sec:definitions}) and then to control the subgraphs
traced by short facial walks. In Section~\ref{sec:discharging} these
two consequences give average facial degree at least eight in the
connected graphs to which we apply discharging. We also show why the
local hypotheses alone cannot give a smaller coefficient.
Section~\ref{sec:reductions} handles leaves, cycles and disconnected
components to complete the proof.

\section{Embeddings and short cycles}\label{sec:definitions}

We use embeddings in the plane and in the projective plane. A cycle in
a surface is \emph{two-sided} if it has an annular neighbourhood, and
\emph{one-sided} if it has a M\"obius-band neighbourhood. A
\emph{parity embedding} is an embedding in which even cycles are
two-sided and odd cycles are one-sided. We recall the following
established thrackle facts.

\begin{theorem}\label{thm:2}
The following hold for thrackleable graphs.
\begin{enumerate}[label=\textup{(\roman*)}]
\item\label{fact:C4} A thrackleable graph contains no four-cycle.
\item\label{fact:plane} A connected, bipartite thrackleable graph has a
plane embedding.
\item\label{fact:parity} A connected nonbipartite thrackleable graph has
a parity embedding in the projective plane.
\end{enumerate}
\end{theorem}

For parts~\ref{fact:C4} and~\ref{fact:plane}, see Lov\'asz, Pach and
Szegedy~\cite[Lemma~2.1 and Theorem~1.4]{LPS1997}.
Part~\ref{fact:parity} follows from
Cairns--Nikolayevsky~\cite[Theorem~2]{CN2009}; the combined embedding
statement is also recorded in Fulek--Pach~\cite[Corollary~1]{FP2019}.
These are crossing-free embeddings of the same abstract graph $G$:
its $n$ vertices and $m$ edges remain unchanged. Crossings of the
original thrackle drawing do not become vertices in these embeddings.
Two distinct faces are called neighbours if they share an edge.

Every subgraph of a thrackleable graph is thrackleable.
Theorem~\ref{thm:pammer} therefore implies that every connected subgraph
on at most twelve vertices contains at most one cycle. We will use the
following three consequences.

\begin{lemma}\label{lem:intersections}
In a thrackleable graph, the following statements hold.
\begin{enumerate}[label=\textup{(\roman*)}]
\item\label{fact:C6} Distinct $6$-cycles have disjoint vertex sets.
\item\label{lem:68} A $6$-cycle and an $8$-cycle cannot share an edge.
\item\label{lem:610} A $10$-cycle and a $6$-cycle cannot share a path
of three or more edges.
\end{enumerate}
\end{lemma}

\begin{proof}
The union of two distinct cycles sharing a vertex is connected and has
more edges than vertices. By Theorem~\ref{thm:pammer}, such a union
must have at least thirteen vertices.

Two $6$-cycles sharing a vertex have at most $11$ vertices in their
union, proving~\ref{fact:C6}. A $6$-cycle and an $8$-cycle sharing
an edge have at most $6+8-2=12$ vertices in their union,
proving~\ref{lem:68}. Finally, if a $10$-cycle and a $6$-cycle shared
a path of three edges, their union would have at most $6+10-4=12$
vertices, proving~\ref{lem:610}. These arguments also allow further
common vertices or edges.
\end{proof}

Part~\ref{fact:C6} was also proved in Fulek--Pach~\cite{FP2011}.

\section{Short faces and the discharging argument}\label{sec:discharging}

Our main estimate keeps track of the Euler characteristic of the
surface in which we embed the graph.

\begin{lemma}\label{lem:main}
Let $G$ be a connected thrackleable graph of minimum
degree at least two, and suppose that $G$ is not a cycle. Set $\chi=2$
if $G$ is bipartite and $\chi=1$ otherwise. Then
\[
 m\le\left\lfloor\frac{4(n-\chi)}3\right\rfloor
 \le\left\lfloor\frac{4(n-1)}3\right\rfloor.
\]
\end{lemma}

For the proof of Lemma~\ref{lem:main}, let $G$ satisfy its hypotheses. By Theorem~\ref{thm:2}\ref{fact:plane}
and~\ref{fact:parity}, choose a plane embedding if $G$ is bipartite,
and a parity embedding in the projective plane otherwise. We view the
plane embedding on the sphere, so that the outer face is treated in
the same way as the other faces.

Both embeddings are \emph{cellular}, meaning that every face is an open
disk. This is standard for a connected graph embedded in the sphere.
For the projective plane, take a one-sided odd cycle, add tree
branches to obtain a connected spanning unicyclic subgraph, and then add
the remaining embedded edges. The unicyclic subgraph has a disk as its
complementary face, and each additional edge splits a disk face into two
disk faces.

Let $f$ be the number of faces in that embedding, and let $d(F)$ be the
length of a boundary walk of face $F$, counted with multiplicity. A
boundary walk may repeat vertices or edges; an edge is counted twice
if both its sides belong to the same face. Using Euler's formula and
double counting gives
\begin{equation}\label{eq:euler}
 n-m+f=\chi
 \qquad\text{and}\qquad
 \sum_F d(F)=2m.
\end{equation}
Our goal is to show $2m\ge8f$. We first establish the required
properties of short faces. Call faces of degree $6,8,10$ hexagonal,
octagonal and decagonal respectively.

\begin{lemma}\label{lem:walks}
Every face has even degree at least $6$. The boundary walk of every face
of degree at most $12$ is a simple cycle.
\end{lemma}

\begin{proof}
In the bipartite case, every closed walk has even length. In a parity
embedding into the projective plane, a closed walk has odd length
exactly when following it reverses local orientation. This holds for
simple cycles by Theorem~\ref{thm:2}\ref{fact:parity}, and extends to
all closed walks by decomposing them into simple cycles and doubled
edges. A facial walk is the boundary walk of a disk face and hence is
null-homotopic, even when it repeats vertices or edges. It therefore
preserves local orientation. Thus all facial lengths are even in either case.

A facial walk cannot immediately reverse an edge, since that would
require a vertex of degree $1$.
Indeed, the local sector between the incoming and outgoing edge ends
would have to be bounded by the same edge end on both sides, which can
only happen at a vertex of degree one.

Let $W$ be a facial walk of length at most $12$, and let $H$ be the
subgraph consisting of its edges and vertices, without multiplicity.
This subgraph is connected and has minimum degree at least two: a
degree-one vertex in $H$ would force an immediate reversal in $W$.
Also $|V(H)|\le12$. Since $H$ is thrackleable,
Theorem~\ref{thm:pammer} gives $|E(H)|\le|V(H)|$. Minimum degree two
gives the reverse inequality, so $H$ is a single cycle.

The nonbacktracking walk $W$ goes around this cycle once or twice, since
each edge has only two sides and a face uses each side at most once.
If it went around twice, it would use both sides of every cycle edge.
At each vertex, the two cycle edge ends would have to be consecutive
in both directions in the local cyclic order: the walk turns from one
to the other on each side. Thus no additional edge could be incident
to a cycle vertex. But then since $G$ is connected, this means the whole graph must be $H$, since no other vertices can be connected to it on either side. This case is excluded in the hypotheses of Lemma \ref{lem:main}, so $W$ must not traverse the cycle twice. Therefore $W$ is a simple cycle.

A face of degree two would give a cycle of length two,
which is impossible in a simple graph; a face of degree four would give
a four-cycle, forbidden by Theorem~\ref{thm:2}\ref{fact:C4}. This proves the lemma.
\end{proof}

Different hexagonal faces have different boundary cycles. Otherwise both
sides of the same cycle would be facial, forcing all its vertices to
have degree two and the connected graph to be that cycle. Consequently,
Lemma~\ref{lem:intersections}\ref{fact:C6} implies that distinct hexagonal faces
have disjoint vertex sets.

We can now bound the number of edges across which a decagonal face can donate charge.

\begin{lemma}\label{lem:decagon}
For every decagonal face $F$, at most $6$ of its edges have a hexagonal
face on the opposite side.
\end{lemma}

\begin{proof}
By Lemma~\ref{lem:walks}, $F$ has a simple cycle boundary. Call an edge
of this cycle covered if a hexagonal face lies on its opposite side.
By Lemma~\ref{lem:intersections}\ref{lem:610}, any hexagonal face covers at most two
consecutive edges on the decagon. Two consecutive edges of the decagon
cannot be covered by different hexagons, as any pair of hexagons must
be vertex-disjoint. Thus every three consecutive edges of the decagon include an uncovered
edge. Among the ten cyclic triples of consecutive edges, each contains
at most two covered edges, and each covered edge belongs to exactly
three triples. If $c$ edges are covered, then $3c\le20$, so $c\le6$.
\end{proof}

\begin{proof}[Proof of Lemma~\ref{lem:main}]
Assign initial charge $d(F)$ to every face. Transfer $1/3$ across each
edge incident with a hexagonal face, from the opposite face to that
hexagon. Hexagonal faces are pairwise vertex-disjoint, so none donates
to another hexagon. Each hexagon gets charge $8$ by receiving $1/3$
along each of its six edges. Octagons have no hexagonal neighbours by
Lemma~\ref{lem:intersections}\ref{lem:68}, so they retain charge $8$. Decagons donate charge
across at most $6$ edges by Lemma~\ref{lem:decagon}, so they retain
charge at least $10-6/3=8$. Every face of degree $d\ge12$ donates at
most $d/3$ and retains at least $2d/3\ge8$.

Thus every face finishes with charge at least $8$. The total charge is
unchanged and equals $2m$, so $8f\le2m$. Equation~\eqref{eq:euler} now gives
\[
 \frac m4\ge f=\chi+m-n,
 \qquad\text{which implies}\qquad
 m\le\frac{4(n-\chi)}3.
\]
Since $m$ is an integer, rounding down proves the lemma.
\end{proof}

The proof of Lemma~\ref{lem:main} uses thrackleability only to obtain
the prescribed embedding, exclude four-cycles, and ensure that each
connected subgraph on at most twelve vertices contains at most one
cycle. The last condition is the local consequence of
Theorem~\ref{thm:pammer}. Even for bipartite plane graphs, these
hypotheses alone do not give a smaller coefficient.

\begin{proposition}\label{prop:grid}
There is a sequence of bipartite plane graphs of girth eight in which
every connected subgraph on at most twelve vertices contains at most
one cycle, and for which the edge-to-vertex ratio tends to $4/3$.
\end{proposition}

\begin{proof}
For $r\ge3$, start with the $r$-by-$r$ square grid and subdivide every
edge once. The resulting graph is bipartite, has girth eight, and has
$3r^2-2r$ vertices and $4r(r-1)$ edges. It remains to verify the
condition on small subgraphs.

Suppose a connected subgraph contains two distinct cycles. Delete
degree-one vertices until none remains. The remaining subgraph still
contains both cycles, and every subdivision vertex in it has both its
incident edges. It is therefore the subdivision of a connected subgraph
$J$ of the original grid, with $v$ vertices and $e\ge v+1$ edges. We
claim that $v\ge6$. Indeed, a simple bipartite graph on at most four
vertices has at most as many edges as vertices. On five vertices the
only bipartite graph with at least six edges is $K_{2,3}$, which cannot
occur in the square grid: two grid vertices have at most two common
neighbours. The subdivision of $J$ thus has $v+e\ge2v+1\ge13$
vertices. Finally, $4r(r-1)/(3r^2-2r)\to4/3$, as claimed.
\end{proof}

\section{Leaves, cycles and disconnected components}\label{sec:reductions}

\begin{proof}[Proof of Theorem~\ref{thm:1}]
First suppose $G$ is connected. If it is a tree or is unicyclic, then
$m\le n\le4(n-1)/3$ for $n\ge4$, as required. Otherwise repeatedly
delete leaves to obtain its connected $2$-core $H$. Write
$h=|V(H)|$ and $e=|E(H)|$. The graph $H$ has minimum degree at least
two and is not a cycle, so $h\ge4$ and $e>h$.

Lemma~\ref{lem:main} gives $e\le4(h-1)/3$, whether or not $H$
contains a triangle. Restoring the deleted leaves gives
\[
 m=e+(n-h)\le\frac{4(h-1)}3+(n-h)
   =\frac{4(n-1)}3-\frac{n-h}3
   \le\frac{4(n-1)}3.
\]

Finally, suppose $G$ has $r\ge2$ connected components, with orders
$n_1,\ldots,n_r$ and edge counts $m_1,\ldots,m_r$. The connected bound
just proved implies $m_i\le4n_i/3-1$ when $n_i\ge4$; the same
inequality holds directly for every simple graph on one, two or three
vertices. Therefore
\[
 m=\sum_{i=1}^{r}m_i\le\frac43\sum_{i=1}^{r}n_i-r
   =\frac43n-r\le\frac43n-2\le\frac{4(n-1)}3.
\]
Since $m$ is an integer, rounding down proves the theorem.
\end{proof}

\paragraph{Funding.}
This research was supported by the European Research Council under
Advanced Grant GeoScape no.~882971.

\paragraph{Declaration of competing interest.}
The author declares no competing interests.

\paragraph{Research methods and data availability.}
During the research, OpenAI's ChatGPT assisted in exploring discharging
rules, checking the argument for short facial walks, and developing the
square-grid example. The only computational result
used in the proof is Pammer's verification
\cite[Observation~5.16]{Pammer2014}. With AI assistance, we independently verified Pammer’s conclusion for graphs on at most twelve vertices by checking all 126 possible minimal counterexample types. The source code, verification certificates, and instructions for reproducing the computation are available at \href{https://github.com/DanielGSimon/Conway-s-thrackle-conjecture-12-vertex/releases/tag/v1.0}{GitHub release (version 1.0)}.

\paragraph{Declaration of generative AI and AI-assisted technologies
in the manuscript preparation process.}
The author wrote the initial manuscript by hand and used 
ChatGPT to transcribe it into \LaTeX, locate and check references,
and revise parts of the exposition. The author takes full
responsibility for the mathematical claims and the final text.

\end{document}